\documentclass[11pt,a4paper]{amsart}

\usepackage{xspace}
\theoremstyle{plain}
\usepackage{amsbsy,amssymb,amsfonts,amstext,latexsym}
\usepackage{mathtools}
\usepackage[utf8]{inputenc}
\usepackage{xcolor}

\usepackage{hyperref}

\usepackage{pgf,tikz}
\usepackage{mathrsfs}
\usetikzlibrary{arrows}

\usepackage{enumerate}
\usepackage{graphics}
\usepackage{subcaption}

\date{\today}
\title{Common hypercyclic algebras for families of products of backward shifts}
\author{Fernando Costa Jr.}
\address{
Université d’Avignon et des Pays de Vaucluse, Laboratoire de Mathématiques, Campus Jean-Henri Fabre, 301, rue Baruch de Spinoza, BP 21239, 84 916 Avignon Cedex 9 France
}
\thanks{The author was partially supported by the grant ANR-17-CE40-0021 of the French National Research Agency ANR (project Front)\\
Email: \href{mailto:fernando.vieira-costa-junior@univ-avignon.fr}{fernando.vieira-costa-junior@univ-avignon.fr} \\
Permanent email: \href{mailto:fv.costajunior@gmail.com}{fv.costajunior@gmail.com}}



\newcommand{\veps}{\varepsilon}

\def\RR{\mathbb R}

\def\NN{\mathbb N}

\def\CC{\mathbb C}

\def\D{\mathcal D}

\newtheorem{theorem}{Theorem}[section]

\newtheorem{lemma}[theorem]{Lemma}

\newtheorem{proposition}[theorem]{Proposition}

\newtheorem{corollary}[theorem]{Corollary}

{\theoremstyle{definition}}
{\theoremstyle{definition}}

{\theoremstyle{definition}\newtheorem{exa}[theorem]{Example}}

{\theoremstyle{definition}}

{\theoremstyle{definition}}

{\theoremstyle{definition}}

\newtheorem{question}[theorem]{Question}

\allowdisplaybreaks

\begin{document}

\definecolor{zzttqq}{rgb}{0.6,0.2,0}
\definecolor{zzttqqa}{rgb}{0.2,0.6,0}
\definecolor{cqcqcq}{rgb}{0.7529411764705882,0.7529411764705882,0.7529411764705882}

\maketitle

\begin{abstract}
In this paper, we generalize to the context of algebras some recent results 
on the existence of common hypercyclic vectors for families of products of backward shift operators. We also give, in a multi-dimensional setting, a positive answer to a question raised by F. Bayart, D. Papathanasiou and the author about the existence of a common hypercyclic algebra on $\ell_1(\NN)$ with the convolution product for the family of backward shifts $(B_{w(\lambda)})_{\lambda>0}$ induced by the weights $w_n(\lambda)=1+\lambda/n$.
\end{abstract}

\section{Introduction}

The study of hypercyclicity consists of analysing the behaviour of the iterates of an operator $T$ on a vector $x$ in a topological vector space $X$. The set of iterates is called the orbit of $x$ under $T$ and denoted by $Orb(T;x):=\{T^n(x) : n\geq 1\}$. When the latter is dense in $X$, we say that $T$ is a \emph{hypercyclic operator} and that $x$ is one of its \emph{hypercyclic vectors}. The set of hypercyclic vectors of $T$ is denoted by $HC(T)$ (see \cite{BM09} or \cite{GePeBook} to learn more on the general theory of linear dynamical systems). When we are dealing with a family $(T_\lambda)_{\lambda\in\Lambda}$ of hypercyclic operators acting on the same space $X$, it is natural to ask whether there is a single vector $x\in X$ which is hypercyclic for each member $T_\lambda$ of the family. We usually assume that the parameter set $\Lambda$ is $\sigma$-compact and that $(\lambda,x)\mapsto T_\lambda(x)$ is continuous from $\Lambda\times X$ into $X$. Such a vector is called \emph{common hypercyclic vector} for the continuous family $(T_\lambda)_{\lambda\in\Lambda}$, which in turn is called \emph{common hypercyclic family} of operators.

This question is interesting and non-trivial. On the one hand, many classical family of operators are common hypercyclic, for example $(\lambda B)_{\lambda>1}$ on $\ell_p(\NN)$, $p\in[1,+\infty)$, where $B$ is the backward shift, or $(T_a)_{a\neq 0}$ on $H(\CC)$, where $T_a:f\mapsto f(\cdot + a)$ is the operator of translation by $a\neq 0$. On the other hand, an example due to A. Borichev (quoted in \cite{AG}) show that the family $(sB\times tB)_{(s,t)\in \Lambda}$ does not have common hypercyclic vectors on $\ell_2(\NN)\times \ell_2(\NN)$ whenever $\Lambda$ has positive Lebesgue measure. Of course, the smaller the parameter set is, the better is the chance of finding a common hypercyclic vector. 

When $X$ has a structure of algebra, it is natural to ask whether or not $HC(T)\cup \{0\}$ contains a non-trivial subalgebra of $X$, what we shall name as \emph{hypercyclic algebra} for $T$. Although the subject is not new, appearing with a negative result in \cite{ACPS1,ACPS2} and with a positive result in \cite{shk, BM09}, it is continuously catching the attention of many researchers on the field in the last few years. 
In this paper, we are exclusively interested in the case where $X$ is a Fréchet sequence algebra, that is, we will assume that $X$ is a Fréchet subspace of the space of all sequences $\omega=\CC^{\NN}$ and on which a well defined product $\cdot$ satisfies, for all $x,y\in X$ and $q\geq 1$, \[\|x\cdot y\|_q\leq \|x\|_q\times \|y\|_q,\]
where the non-decreasing separating sequence of seminorms $(\|\cdot\|_q)_{q\geq 1}$ induces the complete topology of $X$. We are mainly interested on families of \emph{weighted backward shift} operators, that is, maps $B_w:X\to X$ defined by \[B_w(x_0,x_1,x_2,...)=(w_1x_1,w_2x_2,...), \quad x\in X,\]
where the so called \emph{weight sequence} $w=(w_n)_n$ is a sequence of positive real numbers. We say that $w$ is \emph{admissible} when $B_w$ is continuous.

Two classical products are commonly considered on sequence algebras: the coordinatewise product and the convolution (or Cauchy) product, respectively defined, for any sequences $(a_n)_{n\geq 0}, (b_n)_{n\geq 0}$ in $\omega$, as \[(a_n)_n\cdot (b_n)_n = (a_n b_n)_n\quad\quad\text{and}\quad\quad (a_n)_n\cdot (b_n)_n = \Big(\sum_{k=0}^n a_k b_{n-k}\Big)_n.\] The first one turns $\ell_p(\NN), p\geq 1,$ and $c_0(\NN)$ into Fréchet sequence algebras and the second one does the same with $\ell_1(\NN)$. These products also make $H(\CC)$ a Fréchet sequence algebra when we endow it with the sequence of seminorms $(\|\cdot\|_q)_{q\geq 1}$ given by 
\[\left\|\sum_{n=0}^\infty a_n z^n\right\|_q=\sum_{n=0}^\infty |a_n|q^n.\]

A powerful tool to find common hypercyclic vectors is the well known Costakis-Sambarino criterion (see \cite{CoSa04a}). Its hypothesis give a clear way of constructing the partition of the parameter set required to apply the Basic Criterion for common hypercyclicity \cite[Lemma 7.12]{BM09} (or Theorem \ref{basic:alg} for a version adapted to algebras). On higher dimensions, on the other hand, the construction of the partition is more delicate as there is no trivial way of ordering the sets in the partition. While in one dimension, one can arrange the partition in a way that inequalities such as $\lambda<\lambda'$ can be used to get rid of otherwise problematic factors (see the proof of \cite[Theorem 3.12]{BCP2}), in two dimensions one cannot expect to do anything similar with two parameters $(\lambda,\mu)$ and $(\lambda',\mu')$. Hence, proximity conditions on the elements of the partitions become essential. In this paper, we make use of the ideas in the recent work \cite{BCM21} to obtain similar results for common hypercyclic algebras for families of product of backward shift operators. Let us summarize our main findings.

\subsection{Coordinatewise product}

In the same vein as the results from \cite{BCM21}, we have obtained a more general statement which includes both the following corollaries.

\begin{corollary}\label{corol:prac:1}
Let $d\geq 1$, $X=\ell_p(\NN)$ or $c_0(\NN)$, $\alpha\in(0, 1/d)$, $I\subset (0,+\infty)$ be $\sigma$-compact and $(w(\lambda))_{\lambda\in I}$ be a continuous family of weights. Assume that, for all $I_0\subset I$ compact, there exist $D_1,D_2,D_3>0$ and $N\geq 0$ such that, for all $n\geq N$,
\begin{itemize}
\item $a\in I_0\mapsto \sum_{i=1}^n\log\big(w_i(a)\big)$ is $D_1n^\alpha$-Lipschitz;
\item $\inf_{a\in I_0} w_1(a)\cdots w_n(a)\geq D_2\exp(D_3 n^\alpha)$.
\end{itemize}
Then $\big(B_{w(\lambda_1)}\times \cdots \times B_{w(\lambda_d)}\big)_{\lambda\in I^d}$ admits a common hypercyclic algebra for the coordinatewise product.
\end{corollary}

\begin{corollary}\label{corol:prac:2}
Let $d\geq 1$, $X = \ell_p(\NN)$ or $c_0(\NN)$, $I\subset (0,+\infty)$ be $\sigma$-compact and $(w(\lambda))_{\lambda\in I}$ be a continuous family of weights. Assume that, for all $I_0\subset I$ compact, there exist $D_1,D_2,\gamma > 0$ and $N \geq 0$ such that, for all $n \geq N$,
\begin{itemize}
    \item $a\in I_0\mapsto \sum_{i=1}^n \log(w_i(a))$ is $D_1 \log(n)$-Lipschitz;
    \item $\inf_{a\in I_0}w_1(a)\cdots w_n(a)\geq D_2 n^\gamma.$
\end{itemize}
Then $\big(B_{w(\lambda_1)}\times \cdots \times B_{w(\lambda_d)}\big)_{\lambda\in I^d}$ admits a common hypercyclic algebra for the coordinatewise product.
\end{corollary}

Thus we have found a large class of families of operators parametrized by subsets of $\RR^d$ admitting common hypercyclic algebras. Particularly, for both $w_n(a)=1+\frac{a}{n^{1-\alpha}}$, with $\alpha\in(0,1/d)$, and $w_n(a)=1+\frac{a}{n},$ the family $\big(B_{w(\lambda_1)}\times \cdots \times B_{w(\lambda_d)}\big)_{\lambda\in (0,+\infty)^d}$ admits a common hypercyclic algebra for the coordinatewise product.

As for the characterization \cite[Theorem 2.1]{BCM21}, the same holds true for the hypercyclic algebras for the coordinatewise product under natural modifications. For a particular case of weights $\big(w_n(\lambda)\big)_{\lambda}$ with the property that the function $a\mapsto \sum_{k=1}^n\log\big(w_k(a)\big)$ is $n^\alpha$-Lipschitz, we obtained the following equivalence.

\begin{proposition}
Let $\alpha\in(0,1]$ and $w_1(a)\cdots w_n(a)=\exp(an^\alpha)$ and $X=\ell_p(\NN)$ or $c_0(\NN)$ endowed with the coordinatewise product. The following assertions are equivalent:
\begin{enumerate}[(a)]
    \item $(B_{w(a)}\times B_{w(b)})_{(a,b)\in\Lambda}$ has a common hypercyclic vector in $X^2$;
    \item $(B_{w(a)}\times B_{w(b)})_{(a,b)\in\Lambda}$ has a common hypercyclic algebra in $X^2$.
\end{enumerate}
\end{proposition}

\subsection{Convolution product}

As we might expect, the convolution product is more complicated to be dealt with. With the coordinatewise product, since it satisfies $\big(\sum a_n e_n \big)^m=\sum a_n^me_n$, one has more liberty to play with the coefficients without worrying about changes on the supports, we can easily define $m$-roots for vectors and, as a consequence, the adapted conditions in the criteria for algebras change only for a few new powers in comparison with the original criteria for vectors. On the flip side, the convolution product mixes the supports and is hence less predictable. This behavior usually lead to more complicated criteria for the existence of algebras under this product, which are in general not analogous to classical results for the existence of single vectors.

As remarked by its authors, criterion \cite[Theorem 3.12]{BCP2} for the existence of common hypercyclic algebras fails to apply to the family of weights $w_n(\lambda)=1+\frac{\lambda}{n}$, what can be thought as a surprise as this weight should not be hard to work with. In fact, for both examples given by the authors, which are $(\lambda B)_{\lambda>1}$ and $(\lambda D)_{\lambda>0}$, the functions $\lambda\mapsto \sum_{k=1}^n \log\big(w_k(\lambda)\big)$ are $n$-Lipschitz and the divergence of $\sum 1/n$ is used in the construction of the partition. For the weights $w_n(\lambda)=1+\frac{\lambda}{n}$, however, the function $\lambda\mapsto \sum_{k=1}^n \log\big(w_k(\lambda)\big)$ is $\log(n)$-Lipschitz, and the series $\sum_{n}\frac{1}{\log(n)}$ diverges much faster than $\sum 1/n$. Hence, what one can expect to do is to construct a partition specifically for this last example and to try to use this faster divergence to avoid using \cite[Lemma 3.11]{BCP2}. This method should also work with the family of weights $w_n(\lambda)=(1+\frac{1}{n})^\lambda$ or any other family for which $w_1(\lambda)\cdots w_n(\lambda)$ behaves like $n^\lambda$.

One could think that the fast divergence of $\sum\frac{1}{\log(n)^d}$ could lead to a positive result in higher dimensions. This is indeed true, as we are going to prove in this paper, but not by using the dyadic covering discussed in Section \ref{sec:covering10} below. In fact, this covering doesn't seem to be adapted to $w_n(a)=1+\frac{a}{n}$ on $\ell_1(\NN)^d$, as it contains too many elements. Fortunately, the fact that the family induced by $w_n(a)=1+\frac{a}{n}$ satisfies a $\log(n)$-Lipschitz condition (rather than $n^\alpha$-Lipschitz) will allow us to construct a (trivial) covering adapted to this case and obtain the following result.

\begin{theorem}\label{thm:alglog}
If $w_1(a)\cdots w_n(a)=n^a$ (or if $w_n(a)=1+\frac{a}{n}$ with a few modifications in the proof), then, for all $d\geq 1$, the family $\big(B_{w(\lambda_1)}\times \cdots \times B_{w(\lambda_d)}\big)_{\lambda\in (0,+\infty)^d}$ admits a common hypercyclic algebra on $\ell_1(\NN)^d$ with the convolution product.
\end{theorem}

\section{Coordinatewise product}

In what follows we assume that $X$ is a Fréchet sequence algebra under the coordinatewise product in which $\textrm{span}(e_i)$ is dense, $\Lambda\subset \RR^d$ is a $\sigma$-compact set of parameters and, for each $\lambda=(\lambda_1,...,\lambda_d) \in \Lambda$, the operator $T_\lambda:X^d\to X^d$ is defined by $T_{\lambda}:=B_{w(\lambda_1)}\times\cdots\times B_{w(\lambda_d)}$, where $\big(w(x)\big)_x$ is a continuous family of admissible weights and the map $(\lambda, v)\mapsto T_\lambda(v)$ is continuous. Also, for each $\lambda\in\Lambda$ and $m\in\NN$, we let $\D:=\prod_d\textrm{span}(e_i)$ and define $S_{\lambda, m}:\D\to\D$ by \[S_{\lambda,m}:=F_{w(\lambda_1)^{-1/m}}\times \cdots\times F_{w(\lambda_d)^{-1/m}},\]
where each $F_{w(x)^{-1/m}}$, $x\in I$, is the weighted forward shift on $X$ with weight sequence $\big(w_{n}^{-1/m}(x)\big)_n$. The symbol $\|\cdot\|$ represents the $F$-norm of $X$ and thus satisfies, for all $u\in X$ and $\lambda\in \CC$,
\begin{itemize}
    \item $|\lambda|\leq 1\Rightarrow \|\lambda u\|\leq \|u\|$;
    \item $\lim_{\lambda\to 0}\|\lambda u\| = 0$;
    \item $\|\lambda u\|\leq (|\lambda|+1)\|u\|.$
\end{itemize}
We will make use of the following criterion from \cite{BCP2}.

\begin{theorem}[Basic Criterion for algebras]\label{basic:alg}
Assume that, for every compact $K\subset \Lambda$, every $m'\leq m''$ in $\NN$, every $v\in \mathcal{D}$, every $N>0$ and every $\varepsilon>0$, there exist $q\in\NN$,  $\lambda_0,\dots,\lambda_q\in \Lambda$, sets $\Lambda_0,\dots,\Lambda_q\subset \Lambda$ and positive integers $n_0,\dots,n_q$ with $n_0>N$ and $n_{i+1}-n_i>N$ for $i=0,...,q-1$ such that
\begin{enumerate}[(I)]
\item $\bigcup_i \Lambda_i\supset K$;
\item for all $\lambda \in \Lambda_i$, $i=0,...,q$, and all $m\in[m',m'']$, \[\Bigg\|\sum_j S_{\lambda_j,m'}^{n_j}\big(v^{1/m'}\big)\Bigg\|\leq\varepsilon\quad\quad\text{and}\quad\quad\Bigg\|\sum_{j\ne i} T_{\lambda}^{n_i}\Big(S_{\lambda_j,m'}^{n_j}\big(v^{1/m'}\big)\Big)^m\Bigg\| \leq \varepsilon ; \]
\item  for all $\lambda \in \Lambda_i$, $i=0,...,q$, and all $m\in(m',m'']$, \[\Bigg\|T_\lambda^{n_i}\bigg(\Big(S_{\lambda_i,m'}^{n_i}\big(v^{1/m'}\big)\Big)^m\bigg)\Bigg\|\leq \veps;\]
\item for all $\lambda \in \Lambda_i$, $i=0,...,q$, \[\bigg\|T_{\lambda}^{n_i}\Big(S_{\lambda_i, m'}^{n_i}\big(v^{1/m'}\big)\Big)^{m'}-v\bigg\|\leq\varepsilon.\]
\end{enumerate}
Then $\bigcap_{\lambda\in \Lambda} HC(T_{\lambda})$ contains an algebra (but $0$).
\end{theorem}

\subsection{A practical criterion for algebras}

The following result is a practical criterion which leads to the same applications as found in \cite{BCM21}, but now for hypercyclic algebras.

\begin{theorem}\label{thm:unifcrit}
Let $\alpha\in(0,1/d)$ and let $I\subset (0,+\infty)$ be $\sigma$-compact. Suppose that, for all $m'\in\NN$ and all $I_0\subset I$ compact, there exist $C_1,C_2>0$, $\beta>\alpha d$, $F:\NN\to(0,+\infty)$ with $F(n)\leq C_1 n^\alpha$ and $N_0,M_0>0$ such that, for all $n\geq N_0$,
\begin{enumerate}[(i)]
    \item $a \in I_0 \mapsto\sum_{i=1}^n\log(w_i(a))$ is $F(n)$-Lipschitz;
    \item
    for all $a\in I_0$, $\inf_{a\in I_0}w_{1}(a)\cdots w_{k}(a)\to +\infty$ as $k\to+\infty$;
    \item for all $a\in I_0$ and $k\geq N_0$,  \begin{equation}\label{cond:theo:simple}
        \Bigg\|\frac{\exp\Big(C_2\frac{F(n+k)}{(n+k)^\alpha}k^\alpha\Big)}{w_{1}(a)\cdots w_{k}(a)}e_{k}\Bigg\|\leq \frac{M_0}{k^{\beta}},
    \end{equation}
    \begin{equation}\label{cond2:theo:simple}
        \Bigg\|\frac{1}{\big[w_{1}(a)\cdots w_{k}(a)\big]^{1/m'}}e_{k}\Bigg\|\leq \frac{M_0}{k^{\beta}}.
    \end{equation}
\end{enumerate}
Then $\big(B_{w(\lambda_1)}\times\cdots\times B_{w(\lambda_d)}\big)_{\lambda\in I^d}$ admits a common hypercyclic algebra on $X^d$ for the coordinatewise product.
\end{theorem}

If we assume that $X$ admits a continuous norm, then condition (ii) follows from (iii) so it can be removed. Indeed, if $X$ has a continuous norm then the sequence $(e_n)$ is bounded below.

Before proving this result, we will show that the two practical Corollaries 3.2 and 3.6 from \cite{BCM21} (here stated as Corollaries \ref{corol:prac:1} and \ref{corol:prac:2}) give not only a common hypercyclic vector but also a common hypercyclic algebra for the coordinatewise product. The first corollary is proven as follows.

\begin{proof}[Proof of Corollary \ref{corol:prac:1}] Let $m'\in\NN$ and $I_0\subset I$ compact. By hypothesis there are $D_1,D_2,D_3>0$ and $N\geq 0$ as in the statement. We define $F(n):=D_1 n^\alpha$, $C_1:=D_1$ and fix $C_2<\frac{D_3}{D_1}$ and $\beta>\alpha d$. Also we let $N_0> N$ big enough so that, for all $k\geq N_0$, \[\min\Big\{D_2\exp\big((D_3-C_2D_1)k^\alpha\big),\big[D_2\exp(D_3k^\alpha)\big]^{1/m'}\Big\}\geq k^\beta.\]
Conditions (i) and (ii) are satisfied by hypothesis and, for all $n,k\geq N_0$ and $a\in I_0$, we have
\begin{align*}
    \Bigg\|\frac{\exp\Big(C_2\frac{F(n+k)}{(n+k)^\alpha}k^\alpha\Big)}{w_{1}(a)\cdots w_{k}(a)}e_{k}\Bigg\| 
    &= \frac{\exp\Big(C_2\frac{D_1(n+k)^\alpha}{(n+k)^\alpha}k^\alpha\Big)}{w_{1}(a)\cdots w_{k}(a)}\\
    &\leq \frac{\exp\Big(C_2D_1k^\alpha\Big)}{D_2\exp\big(D_3 k^\alpha\big)}\\
    &= \frac{1}{D_2\exp\big((D_3-C_2D_1)k^\alpha\big)}\\
    &\leq \frac{1}{k^\beta}
\end{align*}
and 
\[
    \Bigg\|\frac{1}{\big[w_{1}(a)\cdots w_{k}(a)\big]^{1/m'}}e_{k}\Bigg\|
    \leq \frac{1}{\big[D_2\exp(D_3k^\alpha)\big]^{1/m'}}\leq \frac{1}{k^\beta},
\]
what verifies (iii) and completes the proof.
\end{proof}

\begin{exa}
Let $d\geq 1$,  $X=\ell_p(\NN)$ or $c_0(\NN)$, $\alpha\in (0,1/d)$ and $(w(\lambda))_{\lambda>0}$ be the family of weights defined by $w_1(\lambda)\cdots w_n(\lambda)=\exp(\lambda n^\alpha)$ \big(or $w_n(\lambda) = 1 + \frac{\lambda}{n^{1-\alpha}}$\big) for all $n\geq 1$. Then the family $\big(B_{w(\lambda_1)}\times\cdots\times B_{w(\lambda_d)}\big)_{\lambda\in(0,+\infty)^d}$ admits a common hypercyclic algebra for the coordinatewise product.
\end{exa}


The second corollary is proven as follows.

\begin{proof}[Proof of Corollary \ref{corol:prac:2}]
Let $m'\in\NN$ and $I_0\subset I$ compact. From the hypothesis there are $D_1,D_2, \gamma >0$ and $N\geq 0$ as in the statement. Fixing $\alpha<\min\big(\frac{\gamma}{d m'},\frac{1}{d}\big)$, one can find $C_1$ big enough so that $D_1\log(n)\leq C_1 n^\alpha$ for all $n\in\NN$. Let $\beta\in\big(\alpha d, \frac{\gamma}{m'}\big)$ and $C_2<\frac{\gamma-\beta}{D_1}$. Finally let $N_0> N$ big enough so that $n\mapsto \frac{\log(n)}{n^\alpha}$ is decreasing on $[N_0,+\infty)$.
It is now easy to apply Theorem \ref{thm:unifcrit} with $F(n)=D_1\log(n)$. Conditions (i) and (ii) are satisfied by hypothesis and, for all $a\in I_0$ and $n,k\geq N_0$,
\[
    \Bigg\|\frac{\exp\Big(C_2\frac{F(n+k)}{(n+k)^\alpha}k^\alpha\Big)}{w_{1}(a)\cdots w_{k}(a)}e_{k}\Bigg\|= \frac{\exp\Big(C_2D_1\frac{\log(n+k)}{(n+k)^\alpha}k^\alpha\Big)}{w_{1}(a)\cdots w_{k}(a)}\leq \frac{k^{C_2D_1}}{D_2k^\gamma}=\frac{D_2^{-1}}{k^{\gamma-C_2D_1}}\leq \frac{D_2^{-1}}{k^\beta}
\]
and
\[
    \Bigg\|\frac{1}{\big[w_{1}(a)\cdots w_{k}(a)\big]^{1/m'}}e_{k}\Bigg\| = \frac{1}{\big[w_{1}(a)\cdots w_{k}(a)\big]^{1/m'}} \leq \frac{1}{\big[D_2k^\gamma\big]^{1/m'}}\leq \frac{D_2^{-1/m'}}{k^\beta},
\]
what verifies (iii) and completes the proof.
\end{proof}

\begin{exa}
Let $d\geq 1$,  $X=\ell_p(\NN)$ or $c_0(\NN)$ and $(w(\lambda))_{\lambda>0}$ be the family of weights defined by $w_n(\lambda) = 1 + \frac{\lambda}{n}$ \big(or  $w_n(\lambda) := \big(1 + \frac{1}{n}\big)^\lambda$\big). Then the family $\big(B_{w(\lambda_1)}\times\cdots\times B_{w(\lambda_d)}\big)_{\lambda\in(0,+\infty)^d}$ admits a common hypercyclic algebra for the coordinatewise product.
\end{exa}

The problem of finding a common hypercyclic \emph{vector} in the limit case $\alpha=1/d$ remains open. Although we don't intend to explore this difficulty, it remains as an interesting open problem. We state it here in the form of a bi-dimensional example.
\begin{question}
Does there exist a set of parameters $\Lambda\subset \RR^2$ with positive Lebesgue measure such that the family $(B_{w(\lambda)}\times B_{w(\mu)})_{(\lambda,\mu)\in\Lambda}$, with $w_n(\lambda)=1+\frac{\lambda}{\sqrt{n}}$, admits a common hypercyclic vector on $c_0(\NN)\times c_0(\NN)$?
\end{question}

\subsubsection{A covering lemma}\label{sec:covering10}
In order to prove that the conditions of Theorem \ref{thm:unifcrit} are enough to apply the Basic Criterion, we need a covering result taken from \cite{BCM21}. The $d$-dimensional version of this lemma holds true but, for simplicity, we will state it for $d=2$. Let us fix $\alpha\in(0,1/2)$, $\beta>2\alpha$ and $D>0$. 
An element $\lambda$ of $\RR^2$ will be written $\lambda=(x,y)$.


\begin{lemma}\label{lem:coveringF}
There exists a constant $c>0$ such that, for all $D>0$ and all compact squares $K\subset \RR^2$ with $\textrm{diam}(K)\leq cD$, then,
for all $\tau>0$, for all $\eta>0$, for all $N\geq 1$, there exist $q\geq 1$, a sequence of integers $(n_j)_{j=0,\dots,q-1}$, a sequence of parameters $(\lambda_j)_{j=0,\dots,q-1}$, a sequence of compact squares $(\Lambda_j)_{j=0,\dots,q-1}$ such that
\begin{enumerate}[(a)]
\item $n_0\geq N$, $n_{j+1}-n_j\geq N$;
\item  $K=\bigcup_{j=0}^{q-1}\Lambda_j$ and, for all $j=0,\dots,q-1$, writing $\lambda_j=(x_j,y_j)$, we have $\Lambda_j\subset \left[x_j,x_j+\frac\tau{n_j^\alpha}\right]\times  \left[y_j,y_j+\frac\tau{n_j^\alpha}\right]$;
\item for all $0\leq j<l\leq q-1$, for all $\lambda\in \Lambda_j$, for all $\mu\in \Lambda_l$, 
$$\|\lambda-\mu\|_\infty \leq \frac{D(n_l-n_j)^\alpha}{n_l^\alpha};$$
\item $\sum_{j=0}^{q-1}\frac 1{n_j^\beta}\leq \eta$;
\item for all $j\in\{0,\dots,q-1\}$, $\sum_{l\neq j}\frac{1}{|n_l-n_j|^\beta}\leq\eta$.
\end{enumerate}
\end{lemma}
This covering works well for families of weights $\big(w_n(a)\big)_{a}$ such that $a\mapsto \sum_{k=1}^n\log\big(w_n(a)\big)$ is $F(n)$-Lipschitz for some function $F:\NN\to\NN$ satisfying $F(n)\leq C n^\alpha$, $C>0$, $0<\alpha<1/d$, giving common hypercyclic vectors for the family of products of $d$ backward shifts. 
As usual, it is not very difficult to transfer these results to the context of hypercyclic algebras for the coordinatewise product, what lead us to the formulation of Theorem \ref{thm:unifcrit}. As we have discussed, the convolution product is more delicate to work with.

\begin{proof}[Proof of Theorem \ref{thm:unifcrit}]
For simplicity we prove the result for $d=2$, but it is clear that same can be done for any $d\in\NN$. Also, in order to shrink the notations, we will write $\widehat{w}_n(x)=w_1(x)\cdots w_n(x)$ for any $x$ and $n$. We aim to apply the Basic Criterion for algebras with $\Lambda=(0,+\infty)^d$ and some compact square $K\subset I_0^d$ where $I_0\subset (0,+\infty)$ is a compact interval. Let $m'\leq m''$ in $\NN$ and $(u,v)\in\D$ with support in $[0,p]$ for some $p\in\NN$, say $u=\sum_{l=0}^pu_le_l$ and $v=\sum_{l=0}^pv_le_l$. By hypothesis we find $C_1,C_2>0$, $\beta>\alpha d$, $F:\NN\to (0,+\infty)$  and $N_0,M_0>0$ as in the statement. Let us fix $D<C_2/2^\alpha$ and define \[A_p:=\sup\big\{\big[\widehat{w}_l(x)\|(u,v)\|_\infty\big]^{m/m'} : m\in[m',m'']\cup\{1\}, l=0,...,p\text{ and }x\in I_0\big\},\]
\[A_p':=\inf\{\widehat{w}_l(x) : l=0,...,p\text{ and }x\in I_0\}.\]We can suppose that $\textrm{diam}(K)\leq cD$, otherwise we subdivide it into smaller squares. In order to verify the basic criterion, let $N>0$ and $\veps>0$ arbitrary. Let $\tau_0>0$ small enough so that
\[\max_{l=0,...,p}\big(\|\tau_0 u_l e_l\|,\|\tau_0 v_l e_l\|\big)\leq\frac{\veps}{2(p+1)}\]
and let $\tau>0$ small enough so that
\[\Big|\exp\big(C_12^\alpha\tau\big)-1\Big|\leq \frac{\tau_0}{A_p/A_p'+\|(u)\|_\infty}.\]
Finally let us fix $M>\max\{N_0,N,p\}$ big enough so that, for all $n\geq M$, $l=0,...,p$, $m\in[m',m'']$ and $a\in I_0$, we have $\widehat{w}_{l+n}(a)\geq 1$ and \[\Bigg\|\frac{1}{\big(\widehat{w}_{l+n}(a)\big)^{1/m'}}e_l\Bigg\|\leq \frac{\veps}{2(p+1)\big((A_p/A_p')\exp(\tau C_1 2^\alpha)+1\big)}.\] We then apply Lemma \ref{lem:coveringF} for $\tau>0$, for 
\[\eta=\frac{\veps}{2M_0(p+1)(A_p+1)}\]
and for $M$ and find $q\geq 1$, integers $n_0,...,n_{q-1}$, parameters $\lambda_0,...,\lambda_{q-1}$ and squares $\Lambda_0,...,\Lambda_{q-1}$ such that properties (a)-(e) hold true. We claim that conditions (I)-(IV) of Basic Criterion are satisfied.

Condition (I) is automatically true. For condition (II), let $\lambda\in\Lambda_i$ for some $i\in\{0,...,q\}$, and let $m\in[m',m'']$. We have
\begin{align*}
    \Bigg\|\sum_{j=0}^{q-1}S_{\lambda_j,m'}^{n_j}(u,v)\Bigg\|
    &=\Bigg\|\sum_{j=0}^{q-1}F_{w(x_j)^{-1/m'}}^{n_j}(u^{1/m'})\Bigg\| + \Bigg\|\sum_{j=0}^{q-1}F_{w(y_j)^{-1/m'}}^{n_j}(v^{1/m'})\Bigg\|.
\end{align*}
Let us treat each parcel separately. Firstly we apply (\ref{cond2:theo:simple}) and get
\begin{align*}
    \Bigg\|\sum_{j=0}^{q-1}F_{w(x_j)^{-1/m'}}^{n_j}(u^{1/m'})\Bigg\|
    &= \Bigg\|\sum_{j=0}^{q-1}\sum_{l=0}^p \frac{u_l^{1/m'}}{\big[w_{l+1}(x_j)\cdots w_{l+n_j}(x_j)\big]^{1/m'}}e_{l+n_j} \Bigg\| \\
    &\leq \sum_{j=0}^{q-1}\sum_{l=0}^p\big(|\widehat{w}_l(x_j)u_l|^{1/m'}+1\big)\Bigg\| \frac{1}{\big[\widehat{w}_{l+n_j}(x_j)\big]^{1/m'}}e_{l+n_j}\Bigg\| \\
    &\leq \sum_{j=0}^{q-1}\sum_{l=0}^p(A_p+1)\frac{M_0}{(l+n_j)^\beta} \\
    &\leq \sum_{j=0}^{q-1}\frac{(p+1)(A_p+1)M_0}{n_j^\beta} \\
    &\leq (p+1)(A_p+1)M_0\eta \\
    &\leq \frac{\veps}{2}.
\end{align*}
Analogously we get
$\big\|\sum_{j=0}^{q-1}F_{w(y_j)^{-1/m'}}^{n_j}(v^{1/m'})\big\|\leq \frac{\veps}{2},$
what proves the first part of condition (II). For the second, writing $\lambda=(x,y)$, we have
\begin{align*}
\Bigg\|\sum_{j\ne i} T_{\lambda}^{n_i}\Big(S_{\lambda_j,m'}^{n_j}\big((u,v)^{1/m'}\big)\Big)^m\Bigg\|&\leq\Bigg\|\sum_{j\ne i} B_{w(x)}^{n_i}\Big(F_{w(x_j)^{-1/m'}}^{n_j}\big(u^{1/m'}\big)\Big)^m\Bigg\|\\ &\quad\quad\quad\quad\quad+\Bigg\|\sum_{j\ne i} B_{w(y)}^{n_i}\Big(F_{w(y_j)^{-1/m'}}^{n_j}\big(v^{1/m'}\big)\Big)^m\Bigg\|.
\end{align*}
As before, we treat these parcels separately and get
\begin{align*}
    \Bigg\|\sum_{j\ne i} B_{w(x)}^{n_i}\Big(F_{w(x_j)^{-1/m'}}^{n_j}&\big(u^{1/m'}\big)\Big)^m\Bigg\| \\
    &= \Bigg\|\sum_{j>i} B_{w(x)}^{n_i}\Big(F_{w(x_j)^{-1/m'}}^{n_j}\big(u^{1/m'}\big)\Big)^m\Bigg\|\\
    &= \Bigg\|\sum_{j> i} \sum_{l=0}^pu_l^{m/m'}\frac{w_{l+n_j-n_i+1}(x)\cdots w_{l+n_j}(x)}{\big[w_{l+1}(x_j)\cdots w_{l+n_j}(x_j)\big]^{m/m'}}e_{l+n_j-n_i}\Bigg\|\\
    &\leq \sum_{j> i} \sum_{l=0}^p\Bigg\|u_l^{m/m'}\frac{w_{l+n_j-n_i+1}(x)\cdots w_{l+n_j}(x)}{\big[w_{l+1}(x_j)\cdots w_{l+n_j}(x_j)\big]^{m/m'}}e_{l+n_j-n_i}\Bigg\|.
\end{align*}
Notice that, using $m/m'\geq 1$ and $D<C_2/2^\alpha$,
\begin{align*}
    &\Bigg|u_l^{m/m'}\frac{w_{l+n_j-n_i+1}(x)\cdots w_{l+n_j}(x)}{\big[w_{l+1}(x_j)\cdots w_{l+n_j}(x_j)\big]^{m/m'}} \Bigg|\\*
    &\quad\quad\quad\quad\quad\quad\leq |\widehat{w}_l(x_j)u_l|^{m/m'}
    \frac{\widehat{w}_{l+n_j}(x)}{\widehat{w}_{l+n_j}(x_j)}\times \frac{1}{\widehat{w}_{l+n_j-n_i}(x)}  \\
    &\quad\quad\quad\quad\quad\quad\leq A_p \exp\Bigg(\Bigg|\sum_{i=1}^{l+n_j}\log(w_i(x))-\sum_{i=1}^{l+n_j}\log(w_i(x_j))\Bigg|\Bigg)\times \frac{1}{\widehat{w}_{l+n_j-n_i}(x)}\\
    &\quad\quad\quad\quad\quad\quad\leq  A_p \exp \big(F(l+n_j)|x-x_j|\big)\times \frac{1}{\widehat{w}_{l+n_j-n_i}(x)}\\
    &\quad\quad\quad\quad\quad\quad\leq  A_p\exp \Bigg(D\frac{F(l+n_j)}{n_j^\alpha}(n_j-n_i)^\alpha\Bigg) \times \frac{1}{\widehat{w}_{l+n_j-n_i}(x)}\\
    &\quad\quad\quad\quad\quad\quad\leq  A_p \frac{\exp\Big(D\frac{F(l+n_j)}{(l+n_j)^\alpha}(l+n_j-n_i)^\alpha\times\frac{(l+n_j)^\alpha}{n_j^\alpha}\times\frac{(n_j-n_i)^\alpha}{(l+n_j-n_i)^\alpha}\Big)}{\widehat{w}_{l+n_j-n_i}(x)}\\
    &\quad\quad\quad\quad\quad\quad\leq  A_p \frac{\exp\Big(D2^\alpha \frac{F(l+n_j)}{(l+n_j)^\alpha}(l+n_j-n_i)^\alpha\Big)}{\widehat{w}_{l+n_j-n_i}(x)}\\
    &\quad\quad\quad\quad\quad\quad\leq  A_p \frac{\exp\Big(C_2 \frac{F(l+n_j)}{(l+n_j)^\alpha}(l+n_j-n_i)^\alpha\Big)}{\widehat{w}_{l+n_j-n_i}(x)}.
\end{align*}
Hence,
\begin{align*}
    \Bigg\|\sum_{j\ne i} B_{w(x)}^{n_i}\Big(F_{w(x_j)^{-1/m'}}^{n_j}&\big(u^{1/m'}\big)\Big)^m\Bigg\|\\ 
    &\leq \sum_{j> i} \sum_{l=0}^p\Bigg\|u_l^{m/m'}\frac{w_{l+n_j-n_i+1}(x)\cdots w_{l+n_j}(x)}{\big[w_{l+1}(x_j)\cdots w_{l+n_j}(x_j)\big]^{m/m'}}e_{l+n_j-n_i}\Bigg\|\\
    &\leq \sum_{j> i} \sum_{l=0}^p\Bigg\|A_p \frac{\exp\Big(C_2 \frac{F(l+n_j)}{(l+n_j)^\alpha}(l+n_j-n_i)^\alpha\Big)}{\widehat{w}_{l+n_j-n_i}(x)} e_{l+n_j-n_i}\Bigg\|\\
    &\leq \sum_{j> i} \sum_{l=0}^p\big(A_p+1\big)\Bigg\| \frac{\exp\Big(C_2 \frac{F(l+n_j)}{(l+n_j)^\alpha}(l+n_j-n_i)^\alpha\Big)}{\widehat{w}_{l+n_j-n_i}(x)} e_{l+n_j-n_i}\Bigg\|\\
    &\leq \sum_{j> i} \sum_{l=0}^p\big(A_p+1\big) \frac{M_0}{(l+n_j-n_i)^\beta}\\
    &\leq M_0(p+1)\big( A_p+1\big)\sum_{j>i}\frac{1}{(n_j-n_i)^\beta}\\
    &\leq M_0(p+1)\big(A_p+1\big)\eta\\
    &\leq \frac{\veps}{2}.
\end{align*}
Analogously, $\Bigg\|\sum_{j\ne i} B_{w(y)}^{n_i}\Big(F_{w(y_j)^{-1/m'}}^{n_j}\big(v^{1/m'}\big)\Big)^m\Bigg\|\leq \frac{\veps}{2}$, what shows that condition (II) is satisfied. For condition (III) and (IV), let $\lambda\in\Lambda_i$ for some $i\in\{0,...,q-1\}$. We write $\lambda=(x,y)$. Notice that, for all $l=0,...,p$ and $(z,z_i)\in\big\{(x,x_i), (y,y_i)\big\},$
\begin{align*}
    \frac{\widehat{w}_{l+n_i}(z)}{\widehat{w}_{l+n_i}(z_i)} 
    &\leq \exp\Bigg(\Bigg|\sum_{j=1}^{l+n_i}\log\big(w_j(z)\big)-\sum_{j=1}^{l+n_i}\log\big(w_j(z_i)\big)\Bigg|\Bigg)\\
    &\leq \exp\big(F(l+n_i)|z-z_i|\big)\\
    &\leq \exp\Big(\tau\frac{F(l+n_i)}{n_i^\alpha}\Big)\\
    &\leq \exp\Big(\tau C_1\frac{(l+n_i)^\alpha}{n_i^\alpha}\Big)\\
    &\leq \exp\big(\tau C_1 2^\alpha\big).
\end{align*}
Given $m\in(m',m'']$, we have
\begin{align*}
    \Bigg\|B_{w(x)}^{n_i}\Big(F_{w(x_i)^{-1/m'}}^{n_i}\big(u^{1/m'}\big)\Big)^m\Bigg\| 
    &\!=\! \Bigg\|\sum_{l=0}^pu_l^{m/m'}\frac{w_{l+1}(x)\cdots w_{l+n_i}(x)}{\big[w_{l+1}(x_i)\cdots w_{l+n_i}(x_i)\big]^{\frac{m}{m'}}}e_{l}\Bigg\|\\
    &\!\leq\! \sum_{l=0}^p\Bigg\|\frac{\big(\widehat{w}_l(x_i)u_l\big)^{m/m'}}{\widehat{w}_{l}(x)}\!\times\!\frac{\widehat{w}_{l+n_i}(x)}{\widehat{w}_{l+n_i}(x_i)}\!\times\! \frac{1}{\big[\widehat{w}_{l+n_i}(x_i)\big]^{\frac{m}{m'}-1}} e_{l}\Bigg\|\\
    &\!\leq\! \sum_{l=0}^p\Bigg\|(A_p/A_p')\exp(\tau C_1 2^\alpha) \frac{1}{\big[\widehat{w}_{l+n_i}(x_i)\big]^{\frac{m}{m'}-1}} e_{l}\Bigg\|\\
    &\!\leq\! \big((A_p/A_p')\exp(\tau C_1 2^\alpha)+1\big)\sum_{l=0}^p\Bigg\|\frac{1}{\big[\widehat{w}_{l+n_i}(x_i)\big]^{1/m'}} e_{l}\Bigg\|\\
    &\!\leq\! \frac{\veps}{2}.
\end{align*}
Analogously, $\Bigg\|B_{w(y)}^{n_i}\Big(F_{w(y_i)^{-1/m'}}^{n_i}\big(v^{1/m'}\big)\Big)^m\Bigg\| \leq \frac{\veps}{2}$, what verifies condition (III). Finally, for condition (IV), we first notice that, for any $x,y$ with $|x-y|<\frac{\tau}{n^\alpha}$, we have
\begin{align*}
    \Bigg|\!\Bigg(\!\frac{w_{l+1}(x)\cdots w_{l+n}(x)}{w_{l+1}(y)\cdots w_{l+n}(y)}-1\!\!\Bigg)u_l\Bigg|
    &\!=\!\Bigg|\!\Bigg(\frac{\widehat{w}_{l}(y)}{\widehat{w}_{l}(x)}\!\times\! \frac{\widehat{w}_{l+n}(x)}{\widehat{w}_{l+n}(y)}\!-\!\frac{\widehat{w}_{l}(y)}{\widehat{w}_{l}(x)}\!+\!\frac{\widehat{w}_{l}(y)}{\widehat{w}_{l}(x)}\!-\!1\!\!\Bigg)u_l\Bigg| \\
    &\!\leq\! \Bigg|\frac{\widehat{w}_{l}(y)}{\widehat{w}_{l}(x)}\Bigg|\|u\|_\infty\!\times\!\Bigg|\frac{\widehat{w}_{l+n}(x)}{\widehat{w}_{l+n}(y)}\!-\!1\Bigg|\!+\!\Bigg|\frac{\widehat{w}_{l}(y)}{\widehat{w}_{l}(x)}\!-\!1\Bigg|\|u\|_\infty  \\
    &\!\leq\! \frac{A_p}{A_p'} \Big|\exp\!\big(F(n\!+\!l)|x\!-\!y|\big)\!-\!1\Big|\!+\!\Big|\exp\!\big(F(l)|x\!-\!y|\big)\!-\!1\Big|\|u\|_\infty\\
    &\!\leq\! \frac{A_p}{A_p'} \Big|\exp\!\big(F(n+l)\frac{\tau}{n^\alpha}\big)-1\Big|+\Big|\exp\!\big(F(l)\frac{\tau}{n^\alpha}\big)-1\Big|\|u\|_\infty\\
    &\!\leq\! \frac{A_p}{A_p'} \Big|\exp\!\big(C_1\tau\frac{(n+p)^\alpha}{n^\alpha}\big)-1\Big|+\Big|\exp\!\big(C_1\tau\frac{p^\alpha}{n^\alpha}\big)-1\Big|\|u\|_\infty\\
    &\!\leq\! \frac{A_p}{A_p'} \Big|\exp\!\big(C_12^\alpha\tau\big)-1\Big|+\Big|\exp\!\big(C_1\tau\big)-1\Big|\|u\|_\infty\\
    &\!\leq\! (A_p/A_p'+\|u\|_\infty) \Big|\exp\!\big(C_12^\alpha\tau\big)-1\Big|\\
    &\!\leq\! \tau_0
\end{align*}
Hence,
\begin{align*}
    &\bigg\|T_{\lambda}^{n_i}\Big(S_{\lambda_i, m'}^{n_i}\big(v^{1/m'}\big)\Big)^{m'}-v\bigg\| \\
    &\quad\quad\leq 
    \bigg\|B_{w(x)}^{n_i}\Big(F_{w(x_i)^{-1/ m'}}^{n_i}\big(u^{1/m'}\big)\Big)^{m'}-u\bigg\| + \bigg\|B_{w(y)}^{n_i}\Big(F_{w(y_i)^{-1/ m'}}^{n_i}\big(v^{1/m'}\big)\Big)^{m'}-v\bigg\| \\
    &\quad\quad= \Bigg\|\sum_{l=0}^p\Bigg(\frac{w_{l+1}(x)\cdots w_{l+n_i}(x)}{w_{l+1}(x_i)\cdots w_{l+n_i}(x_i)}-1\Bigg)u
    _le_l\Bigg\| + \Bigg\|\sum_{l=0}^p\Bigg(\frac{w_{l+1}(y)\cdots w_{l+n_i}(y)}{w_{l+1}(y_i)\cdots w_{l+n_i}(y_i)}-1\Bigg)v
    _le_l\Bigg\|\\
    &\quad\quad\leq \sum_{l=0}^p \Big(\big\|\tau_0 u_le_l\big\|+\big\|\tau_0 v_l e_l\big\|\Big)\\
    &\quad\quad\leq \veps.
\end{align*}
This proves the claim and completes the proof.
\end{proof}

{\color{red}
}


\subsection{Characterization of families of products of weighted shifts admitting a
common hypercyclic algebra for the coordinatewise product}

As done in \cite{BCP1}, many results for vectors can be brought to the context of algebras with the coordinatewise product under almost the same conditions, the difference being the presence of some natural powers. Here we repeat this formula with \cite[Theorem 2.1]{BCM21} and obtain Theorem \ref{thm:carac}. For the proof we need the following property which holds true whenever $(e_n)_n$ is an unconditional basis of a Fréchet space $X$.
\begin{enumerate}\label{UNC}
\item[(UB)\!\!] If $(x_n)\in X$ and $(\alpha_n)\in\ell_\infty$, then $(\alpha_n x_n)\in X$. Moreover, for all $\veps>0$, for all  $M>0$, there exists $\delta>0$ such that, 
for all $x\in X$ with $\|x\|\leq \delta$, for all sequence $(\alpha_n)\in\ell_\infty$ with $\|(\alpha_n)\|_\infty\leq M$, then $(\alpha_n x_n)\in X$ and $\|(\alpha_n x_n)\|<\veps$.
\end{enumerate}

\begin{theorem}\label{thm:carac}\index{characterization!of common hypercyclic algebras}
Let $X$ be a Fréchet sequence algebra for the coordinatewise product with a continuous norm and admitting $(e_n)$ as an unconditional basis. Let $I\subset\mathbb R$ be a nonempty interval, let $d\geq 1$ and let $\Lambda\subset I^d$ be $\sigma$-compact. Let $(B_{w(a)})_{a\in I}$ be a family of weighted shifts on $X$ and assume that $a\in I\mapsto w_n(a)$ is non-decreasing with $\displaystyle\inf_{n\in\NN, a\in I} w_n(a)>0$. Also, assume that there exist $F:\mathbb N\to\RR_+$ non-decreasing and $c,C>0$ such that, for all $n\geq 1$, writing $f_n(a)=\sum_{k=1}^n \log(w_k(a))$, 
\[ \forall (a,b)\in I^2,\ cF(n) |a-b|\leq |f_n(a)-f_n(b)|\leq CF(n) |b-a|, \]
\[ \forall (a,b)\in I^2,\ \frac{w_n(a)}{w_n(b)}\geq c. \]
Then the following assertions are equivalent:
\begin{enumerate}[(a)]
\item $(B_{w(\lambda(1))}\times\cdots \times B_{w(\lambda(d))})_{\lambda\in\Lambda}$ admits a common hypercyclic algebra in $X$;
\item For all $m\geq 1$, there exist $u\in X^d$ such that $u^m$ is a common hypercyclic vector for $(B_{w(\lambda(1))}\times\cdots \times B_{w(\lambda(d))})_{\lambda\in\Lambda}$;
\item For all $m\in\NN$, $\tau>0$, for all $N\geq 1$, for all $\veps>0$, for all $K\subset\Lambda$ compact, there exist
$N\leq n_1<n_1+N\leq n_1<\dots<n_{q-1}+N\leq n_{q}$ and $(\lambda_k)_{k=1,\dots,q}\in I^d$ such that
\begin{enumerate}[(i)]
\item $K\subset \cup_{k=1}^q \prod_{l=1}^d \left[\lambda_k(l)-\frac\tau{F(n_k)},\lambda_k(l)\right]$
\item For all $i=1,\dots,d$,
\[ \left\|\sum_{k=1}^q\frac{1}{\big[w_1(\lambda_k(i))\cdots w_{n_k}(\lambda_k(i))\big]^{1/m}}e_{n_k}\right\|<\veps. \]
\item For all $k=1,\dots,q$, for all $i=1,\dots,d$, for all $l=0,...,N$,
\[ \left\|\sum_{j=k+1}^{q}\frac{w_{n_j-n_k+l+1}(\lambda_k(i))\cdots w_{n_j+l}(\lambda_k(i))}{w_{l+1}(\lambda_j(i))\cdots
w_{n_j+l}(\lambda_j(i))}e_{n_j-n_k}\right\|<\veps. \]
\end{enumerate}
\end{enumerate}
\end{theorem}

\begin{proof} To fix ideas, we will prove the case $d=2$. The modifications needed for obtaining the general case are straightforward, so we let them for the reader. The notations used are as follows. Each elements of $I^2$ will be written as a couple $\lambda=(a,b)$ and will induce the operators $T_\lambda=B_{w(a)}\times B_{w(b)}$ and its right inverse $S_\lambda=F_{w^{-1}(a)}\times F_{w^{-1}(b)}$ (here, $F_{w^{-1}(a)}$ is the forward shift induced by $(w_n^{-1}(a))_n$).


The proof of $(a)\Rightarrow(b)$ is trivial. The proof of $(b)\Rightarrow (c)$ is very similar to what is done in \cite{BCM21}, we leave the details to the reader. Let us show that $(c)\Rightarrow(a)$. We aim to apply the Basic Criterion for algebras, so let $\mathcal D\subset X^2$ be the set of couples of vectors with finite support, let $K\subset \Lambda$ be compact, let $m'\leq m''$ and let $v\in\mathcal D$ and $N\in\NN$, say $v=(v(1),v(2))$ with $v(i)=\sum_{l=0}^Nv_l(i)e_l, i=1,2$. Finally we let $\veps>0$, we fix a small $\tau>0$ (conditions later on) and we apply (c) with $m=m'$ in order to obtain sequences $(n_k)_{k=1,\dots,q}$ and $(\lambda_k)_{k=1,\dots,q}$ satisfying (i), (ii) and (iii). We write $\lambda_k=(a_k,b_k)$ and set $\Lambda_k=K\cap ([a_k-\tau/F(n_k);a_k]\times [b_k-\tau/F(n_k);b_k])$ so that $\bigcup_k\Lambda_k\supset K$. We claim that all conditions of the Basic Criterion for algebras are satisfied for $m_k=n_k-N$. Condition (I) is already verified. For later use, we observe that the condition $\inf_n w_n(\tilde a)$ and the unconditionality of $(e_n)_n$ ensures the continuity of $B_{w(\tilde a)^{1/m'}}$. For condition (II), we start by verifying that 
\begin{align*}
    \Bigg\|\sum_{k=0}^{q}S_{\lambda_k,m'}^{m_k}v\Bigg\|
    &=\Bigg\|\sum_{k=0}^{q}F_{w(a_k)^{-1/m'}}^{m_k}(v(1)^{1/m'})\Bigg\| + \Bigg\|\sum_{k=0}^{q}F_{w(b_k)^{-1/m'}}^{m_k}(v(2)^{1/m'})\Bigg\|
\end{align*}
is smaller than $\veps$. For $i=1,2$ we have
\begin{align*}
&\left\|\sum_{k=1}^q F_{w(a_k)^{-1/m'}}^{m_k}(v(i)^{1/m'})\right\| \leq \\ 
&\quad \quad\leq \sum_{l=0}^{N} (|v_l(i)|^{1/m'}+1)\left\|\sum_{k=1}^q \frac{1}{\big[w_{l+1}(a_k)\cdots w_{n_k-(N-l)}(a_k)\big]^{1/m'}}e_{n_k-(N-l)}\right\|.
\end{align*}
Fixing $\tilde a\in I$, we see that 
\begin{align}\label{proof:eqq}\begin{split}
    &\sum_{k=1}^q \frac{1}{\big[w_{l+1}(a_k)\cdots w_{n_k-(N-l)}(a_k)\big]^{1/m'}}e_{n_k-(N-l)}\\ &\quad\quad\quad\quad =
B_{w(\tilde a)^{1/m'}}^{N-l}\left(\sum_{k=1}^q \frac{x_{k,l}}{\big[w_{1}(a_k)\cdots w_{n_k}(a_k)\big]^{1/m'}}e_{n_k}\right)
\end{split}
\end{align}
for some sequence $(x_{k,l})_k\in\ell_\infty$ with 
$\|(x_{k,l})_k\|_\infty\leq \left(\frac{M}c\right)^{N/m'}$,
where \[M=\max\big(\{1\}\cup\{|w_j(a)|:0\leq j\leq N, (a,b)\in K\textrm{ for some }b\}\big).\] Effectively,
\begin{align*}
\sum_{k=1}^q &\frac1{\big[w_{l+1}(a_k)\cdots w_{n_k-(N-l)}(a_k)\big]^{1/m'}}e_{n_k-(N-l)}\\
&\quad\quad\quad\quad=\sum_{k=1}^q \big[w_1(a_k)\cdots w_l(a_k)\big]^{1/m'}\times \left[\prod_{j=n_k-(N-l)+1}^{n_k}\frac{w_j(a_k)}{w_j(\tilde a)}\right]^{1/m'} \\ & \quad\quad\quad\quad\quad\quad\quad\quad\quad\quad\quad\quad\quad\quad\quad\times B_{w(\tilde a)^{1/m'}}^{N-l}\left(\frac{1}{\big[w_{1}(a_k)\cdots w_{n_k}(a_k)\big]^{1/m'}}e_{n_k}\right)\\
&\quad\quad\quad\quad\leq \sum_{k=1}^q M^{N/m'} \left(\frac{1}{c}\right)^{(N-l)/m'}B_{w(\tilde a)^{1/m'}}^{N-l}\left(\frac{1}{\big[w_{1}(a_k)\cdots w_{n_k}(a_k)\big]^{1/m'}}e_{n_k}\right).
\end{align*}
Hence, the first condition in \hyperref[basic:alg]{(II)} follows from the continuity of $B_{w(\tilde a)^{1/m'}}$ and the unconditionality of $(e_n)$. For the second condition in \hyperref[basic:alg]{(II)}, given $\lambda=(a,b)\in\Lambda_k$, for some $k=1,...,q$, and $m\in[m',m'']$, we have
\begin{align*}
\sum_{j\neq k} B_{w(a)}^{m_k}&\left(F_{w^{-1/m'}(a_j)}^{m_j}(v(i)^{1/m'})\right)^m\\
&=\sum_{j=k+1}^q B_{w(a)}^{m_k}\left(F_{w^{-1/m'}(a_j)}^{m_j}(v(i)^{1/m'})\right)^m\\
&=\sum_{j=k+1}^q \sum_{l=0}^{N} \big[v_l(i)\big]^{m/m'}\frac{w_{n_j-n_k+l+1}(a)\cdots w_{n_j-(N-l)}(a)}
{\big[w_{l+1}(a_j)\cdots w_{n_j-(N-l)}(a_j)\big]^{m/m'}}e_{n_j-n_k+l}\\
&=\sum_{j=k+1}^q \sum_{l=0}^{N} \big[v_l(i)\big]^{m/m'}y_{j,l}\frac{w_{n_j-n_k+l+1}(a)\cdots w_{n_j-(N-l)}(a)}
{w_{l+1}(a_j)\cdots w_{n_j-(N-l)}(a_j)}e_{n_j-n_k+l},
\end{align*}
where
\[y_{j,l}=\frac{1}{\big[w_{l+1}(a_j)\cdots w_{n_j-(N-l)}(a_j)\big]^{m/m'-1}}.\]
Since $X$ has a continuous norm, from (ii) and as $m/m'-1\geq 0$ we get that $(y_{j,l})$ is bounded uniformly on $j$ and $l$. In order to apply (iii) we define, for $0\leq l\leq N$ and $j\geq k+1$,
\[\alpha_{j,l} = y_{j,l}\frac{w_{n_j-n_k+l+1}(a)\cdots w_{n_j+l}(a)}{w_{n_j-n_k+l+1}(a_k)\cdots w_{n_j+l}(a_k)}\times\frac{w_{n_j+l-N+1}(a_j)\cdots w_{n_j+l}(a_j)}{w_{n_j+l-N+1}(a)\cdots w_{n_j+l}(a)}\]
and use that 
\[w_{n_j-n_k+l+1}(a)\cdots w_{n_j+l}(a)\leq w_{n_j-n_k+l+1}(a_k)\cdots w_{n_j+l}(a_k)\]
in order to verify that $(\alpha_{j,l})$ is bounded uniformly on $j$ and $l$. Now,
\begin{align*}
\sum_{j\neq k} B_{w(a)}^{m_k}&\left(F_{w^{-1/m'}(a_j)}^{m_j}(v(i)^{1/m'})\right)^m\\
&=\sum_{j=k+1}^q \sum_{l=0}^{N} \big[v_l(i)\big]^{m/m'}\alpha_{j,l}\frac{w_{n_j-n_k+l+1}(a_k)\cdots w_{n_j+l}(a_k)}
{w_{l+1}(a_j)\cdots w_{n_j+l}(a_j)}e_{n_j-n_k+l}.
\end{align*}
Hence, \hyperref[basic:alg]{(II)} follows from the unconditionality of $(e_n)$ if $\veps$ is small enough

In order to prove \hyperref[basic:alg]{(III)} and \hyperref[basic:alg]{(IV)}, we take $\lambda=(a,b)\in\Lambda_k$, for some $k=0,...,q$, we fix $m\in[m',m'']$ and we write, for $i=1,2$,
\begin{align*}
    B_{w(a)}^{m_k}\left(F_{w^{-1/m'}(a_k)}^{m_k}(v(i))^{1/m'}\right)^m
    &=\sum_{l=0}^{N}v_l(i)^{m/m'}\frac{w_{l+1}(a)\cdots w_{n_k-(N-l)}(a)}{\big[w_{l+1}(a_k)\cdots w_{n_k-(N-l)}(a_k)\big]^{m/m'}}e_l.
\end{align*}
Condition \hyperref[basic:alg]{(III)} corresponds to the case $m>m'$, for which we write 
\begin{align*}
    \frac{w_{l+1}(a)\cdots w_{n_k-(N-l)}(a)}{\big[w_{l+1}(a_k)\cdots w_{n_k-(N-l)}(a_k)\big]^{m/m'}}&\leq \frac{1}{\big[w_{l+1}(a_k)\cdots w_{n_k-(N-l)}(a_k)\big]^{m/m'-1}}\\
    &=z_{k,l}\frac{1}{\big[w_{l+1}(a_k)\cdots w_{n_k-(N-l)}(a_k)\big]^{1/m'}},
\end{align*}
where \[z_{k,l}=\frac{1}{\big[w_{l+1}(a_k)\cdots w_{n_k-(N-l)}(a_k)\big]^{\frac{m-1}{m'}-1}}.\]
Once again, since $X$ has a continuous norm, from (ii) and as $\frac{m-1}{m'}-1\geq 0$, we get that $(z_{k,l})$ is bounded uniformly on $k$ and $l$. Now we can use the fact that $(e_n)$ is bounded bellow and repeat what we did with (\ref{proof:eqq}), now with $(x_{k,l})$ having only one non-zero term. Hence \hyperref[basic:alg]{(III)} follows if we choose $\veps>0$ small enough.

Finally, condition \hyperref[basic:alg]{(IV)} corresponds to the case $m=m'$, in which we write, for $i=1,2$,
\begin{align*}
    &\left\|B_{w(a)}^{m_k}\left(F_{w^{-1/m'}(a_k)}^{m_k}(v(i))^{1/m'}\right)^{m'}-v(i)\right\|\\
&\quad\quad\quad\leq\sum_{l=0}^N (|v_l(i)+1|)\left\|
\left(\frac{w_{l+1}(a)\cdots w_{n_k-(N-l)}(a)}{w_{l+1}(a_k)\cdots w_{n_k-(N-l)}(a_k)}-1\right)e_l\right\|
\end{align*}
and it is easy to show that this becomes smaller than $\veps$ provided $\tau>0$ is taken very small and using $|f_n(a)-f_n(b)|\leq CF(n) |a-b|.$
\end{proof}

If we restrict the $m$-th power in the statement to $m=1$, we obtain a characterization of families admitting a common hypercyclic \emph{vector}. Hence, what distinguish vectors from algebras as far as this characterization goes is the presence of the power $1/m$ in (ii). As soon as these conditions are equivalent for a specific family of weights, one has that the family admits a common hypercyclic vector if and only if it admits a common hypercyclic algebra. We then get the following result.

\begin{proposition}
Let $d\geq1$, $\alpha\in(0,1]$, $w_1(a)\cdots w_n(a)=\exp(an^\alpha)$ and $X=c_0(\NN)$ or $X=\ell_p(\NN), p\in[1,+\infty)$. The following assertions are equivalent:
\begin{enumerate}[(a)]
    \item $(B_{w(\lambda(1))}\times\cdots\times B_{w(\lambda(d))})_{\lambda\in\Lambda}$ has a common hypercyclic vector in $X^d$;
    \item $(B_{w(\lambda(1))}\times\cdots\times B_{w(\lambda(d))})_{\lambda\in\Lambda}$ has a common hypercyclic algebra in $X^d$.
\end{enumerate}
\end{proposition}
\begin{proof}
The proof of $(b)\Rightarrow(a)$ is trivial. Let us assume $(a)$ and prove $(b)$. This family then satisfy (c) of Theorem \ref{thm:carac} for $m=1$. Let $m$ be arbitrary, let $\tau>0$, let $N\geq 1$, $\veps>0$ and let $K\subset \Lambda$ compact. There is $N_0$ such that \begin{equation}\label{eq:cond:equiv}\sum_{n=N_0}^{+\infty}\frac{1}{\exp(an^{\alpha})^{1/m}}<\veps.
\end{equation}
From Theorem \ref{thm:carac}(c, $m=1$) there exist $(n_k,\lambda_k)_{k=1,...,q}$ satisfying $(i)$ and $(iii)$ of the same theorem. Condition $(ii)$ for this arbitrary $m$ follows automatically from (\ref{eq:cond:equiv}), what completes the proof.
\end{proof}

This particular case includes, for example, multiples of the backward shift. In view of the results found in \cite{BCM21}, we obtain the following.

\begin{exa}
Any family $\big(\lambda B\times \mu B\big)_{(\lambda,\mu)\in\Gamma}$ acting on $c_0(\NN)$ or $\ell_p(\NN), p\in[1,+\infty),$ and indexed by a Lipschitz curve $\Gamma\subset (1,+\infty)^2$ admits a common hypercyclic vector. Hence, they admit as well common hypercyclic algebras for the coordinatewise product.
\end{exa}


{\color{red}
}

\section{Convolution product}\label{paperD:sec:conv}

We will obtain Theorem \ref{thm:alglog} by applying the following more general key result from \cite{BCP2}, which is an adapted version of Birkoff's hypercyclic theorem for common hypercyclic algebras.

\begin{proposition}\label{prop:chacriterion}
Let $\Lambda$ be a countable union of compact sets and let $(T_\lambda)_{\lambda\in\Lambda}$ be a family of operators in $\mathcal L(X)$ such that the map  $(\lambda,x)\mapsto T_\lambda(x)$ is continuous from $\Lambda\times X$ into $X$. Assume that, for all compact sets $K\subset \Lambda$,
for all $m\geq1$, for all $U,V$ non-empty open subsets of $X$ and for all neighborhood $W$ of $0$, one can find $u\in U$ such that, for all $\lambda\in K$, 
there exists $N\in\NN$ satisfying
$$
\left\{ 
\begin{array}{l}
 \displaystyle T_\lambda^N(u^n)\in W\textrm{ when }n\leq m-1, \\
 T_\lambda^N(u^{m})\in V.
\end{array}
\right.$$
Then the set of points generating a common hypercyclic algebra for the family $(T_\lambda)_{\lambda\in\Lambda}$ is residual in $X$.
\end{proposition}

As for the covering, we didn't manage to apply Lemma \ref{lem:coveringF} to $w_n(\lambda)=1+\frac{\lambda}{n}$. As this result was made having in mind the case $w_1(a)\cdots w_{n}(a)=\exp(an^\alpha)$, it well fits this setting and naturally grants a common hypercyclic vector. The case $w_n(\lambda)=1+\frac{\lambda}{n}$ is rather different. Now, the function $a\mapsto \sum_{k=1}^n\log\big(w_k(a)\big)$ is $\log(n)$-Lipschitz and a partition made of cubes of side $\frac{\tau}{n_k^\alpha}$ doesn't seem to work well. To overcome this apparent incompatibility we will insist on defining cubes of side $\frac{\tau}{\log(n_k)}$ instead.

In one dimension, we could define something like $\lambda_{j+1}=\lambda_j+\frac{\tau}{\log(n_k)}$ and proceed with similar calculations as in \cite[Theorem 3.12]{BCP2}. The tricky part being to define a suitable sequence $(n_k)_k$, but once it is done, this gives a common hypercyclic algebra for $\big(B_{w(\lambda)}\big)_{\lambda>0}$ and answer \cite[Question 3.16]{BCP2} in the affirmative. What interest us, however, is the possibility of finding such an algebra for $\big(B_{w(\lambda_1)}\times \cdots \times B_{w(\lambda_d)} \big)_{\lambda\in (0,+\infty)^d}$ for any $d\geq 1$. {\color{red} 

}

As we have done before, we will prove the result for $d=2$, but it is clear that the same can be done for any $d\geq 1$.

\begin{proof}[Proof of Theorem \ref{thm:alglog}]
Let $\big(T_{\lambda,\mu}\big)_{(\lambda,\mu)\in (0,+\infty)^2}$ defined by $T_{\lambda,\mu}:=B_{w(\lambda)}\oplus B_{w(\mu)}$ on $X\times X$, where $X=\ell_1(\NN)$ and $w_n(\lambda)=1+\frac{\lambda}{n}$. We aim to apply Proposition \ref{prop:chacriterion} to this family. 

We fix $[a',b']\times[a'',b'']\subset (0,+\infty)^2$ with $b'<2a'$ and $b''<2a''$ and let $m\in \NN$. We also fix an integer $r>\max\{\frac{1}{a'},\frac{1}{a''},1\}$ (as we will see, one could just take $r=1$ in the case $\min\{a',a''\}>1$). As the products $w_1(x)\cdots w_n(x)$ behave like $n^x$, is it enough to prove the claim for $w_1(x)\cdots w_n(x)=n^x$. This will simplify the proof as there will be less constant factors to be dealt with.

For $m=1$ the proposition is equivalent to the existence of a common hypercyclic vector (which is known to be true from \cite{BCM21}). Let us assume $m\geq 2$. Let $U,V,W\subset X\times X$ open and non-empty, with $0\in W$. Let $(u_1,u_2)\in U$, $(v_1,v_2)\in V$, $p\in\NN$ and $\eta>0$ such that $B\big((u_1,u_2);2\eta\big)\subset U$, $B\big((v_1,v_2);3\eta\big)\subset V$ and $\max\text{supp}(u_1,u_2,v_1,v_2)\leq p$. Since for each $n\in\NN$ the function $a\mapsto \sum_{k=1}^n\log(w_k(a))$ is $\log(n)$-Lipschitz and $p$ is fixed, there exists a constant $C_v$ such that, for all $l=0,...,p$ and all $a,b> 0$,
\begin{equation}\label{ineq:lips}
    \Bigg|\sum_{j=l+1}^{n+l}\log\big(w_j(a)\big)-\sum_{j=l+1}^{n+l}\log\big(w_{j}(b)\big)\Bigg|\leq C_v\log(n)|a-b|.
\end{equation}
Let us write $v_i=\sum_{l=0}^pv_{i,l}e_l$ for $i=1,2$ and choose $\sigma$ a very big integer of the form $n^m$ (conditions on its size later). We can choose $\sigma$ such that $\lfloor(\log\sigma)^3+1\rfloor$ is the square of some integer $q$. We then have $\sqrt{q}\in\NN$ and \[(\log\sigma)^3\leq q\leq (\log\sigma)^3+1.\]
We now define $a=\min\{a',a''\}$, $b=\max\{b',b''\}$, and we consider a covering of $[a,b]^2$ by $q$ squares $\Lambda_1,...,\Lambda_q$ of side $\frac{b-a}{\sqrt{q}}$ and centers $(\lambda_1,\mu_1),...,(\lambda_q,\mu_q)$. We choose the sequence of powers \[N_j=(m-1)\sigma+\sigma^{\frac{m-1}{m}}(j+1)^r, \quad j=1, ...,q.\]
Since $\sigma$ is of the form $n^m$, the power $N_j$ is a positive integer for all $j=1,...,q$. Let \[(u',u'')=(u_1,u_2)+\sum_{j=1}^{q}\sum_{l=0}^p\big(d_{l,j}'e_{N_j-(m-1)\sigma+l},d_{l,j}''e_{N_j-(m-1)\sigma+l}\big) + (\veps_1 e_\sigma,\veps_2 e_{\sigma}),\]
where
\[d_{l,j}':=\frac{v_{1,l}}{m\veps_1^{m-1}w_{l+1}(\lambda_j)\cdots w_{l+N_j}(\lambda_j)},\quad d_{l,j}'':=\frac{v_{2,l}}{m\veps_2^{m-1}w_{l+1}(\mu_j)\cdots w_{l+N_j}(\mu_j)},\]
\[\veps_1:=\left(\frac{1}{w_1(a')\cdots w_{m\sigma}(a')}\right)^{\frac{1}{m}},\quad \veps_2:=\left(\frac{1}{w_1(a'')\cdots w_{m\sigma}(a'')}\right)^{\frac{1}{m}}.\]
Let us first show that $(u',u'')\in U$ if $\sigma$ is big enough. We can calculate it coordinate by coordinate and get, for some constant $C_v'>0$ depending on $v$,
\begin{align*}
\|u'-u_1\|_1
    &\leq \left\|\sum_{j=1}^{q}\sum_{l=0}^p d_{l,j}'e_{N_j-(m-1)\sigma+l}\right\|_1 +\|\veps_1 e_{\sigma}\|_1\\
    &\leq \sum_{j=1}^{q}\sum_{l=0}^p \frac{|v_{1,l}|}{m}\times\frac{\big[w_1(a')\cdots w_{m\sigma}(a')\big]^{\frac{m-1}{m}}}{w_{l+1}(\lambda_j)\cdots w_{l+N_j}(\lambda_j)}+\veps_1\\
    &\leq C_v'\sum_{j=1}^{q}\frac{\big(w_1(a')\cdots w_{m\sigma}(a')\big)^{\frac{m-1}{m}}}{w_{1}(a')\cdots w_{N_j}(a')}+\Bigg(\frac{1}{w_1(a')\cdots w_{m\sigma}(a')}\Bigg)^{\frac{1}{m}}\\
    &= C_v' \sum_{j=1}^{q} \left(\frac{(m\sigma)^{\frac{m-1}{m}}}{(m-1)\sigma+\sigma^{\frac{m-1}{m}}(j+1)^r}\right)^{a'}+\Bigg(\frac{1}{w_1(a')\cdots w_{m\sigma}(a')}\Bigg)^{\frac{1}{m}}\\
    &\leq C_v' \sum_{j=1}^{+\infty} \left(\frac{m^{\frac{m-1}{m}}}{(m-1)\sigma^{1/m}+(j+1)^r}\right)^{a'}+\Bigg(\frac{1}{w_1(a')\cdots w_{m\sigma}(a')}\Bigg)^{\frac{1}{m}},
\end{align*}
and this goes to zero as $\sigma\to+\infty$, hence $\|u'-u_1\|< \eta$ if $\sigma$ is big enough. Analogously $\|u''-u_2\|< \eta$, what proves that $(u',u'')\in U$.

Let us now show that, for all $(\lambda,\mu)\in[a',b']\times[a'',b'']$, there is $N$ satisfying
\begin{equation}\label{cond:dim2:t}
\left\{ 
\begin{array}{l}
 \displaystyle T_{\lambda,\mu}^N\big((u',u'')^n\big)\in W\textrm{ for }n<m,\\
 T_{\lambda,\mu}^N\big((u',u'')^m\big)\in V.
\end{array}
\right.
\end{equation}
Given $(\lambda,\mu)\in[a',b']\times[a'',b'']$, let $i\in\{1,...,q\}$ such that $(\lambda,\mu)\in \Lambda_i.$ We then choose $N=N_i$. Let us verify the conditions in (\ref{cond:dim2:t}).

We first notice that, if $n<m$, since $\max\text{supp}(u',u'')^n\leq (m-1)\sigma<N_i=N$ if $\sigma$ is big enough, we have $T_{\lambda,\mu}^N\big((u',u'')^n\big)=0\in W.$ For the $m$-th power we get
\begin{align*}
(u',u'')^m 
    &= (P_0',P_0'')+
    \sum_{j=1}^{q}\sum_{l=0}^p\big(m\veps_1^{m-1}d_{l,j}'e_{N_j+l},m\veps_2^{m-1}d_{l,j}''e_{N_j+l}\big)+(\veps_1^{m}e_{m\sigma},\veps_2^me_{m\sigma}),
\end{align*}
where \[\text{supp}(P_0',P_0'')\subset[0,(m-1)\sigma+p]\cup[0,(m-2)\sigma+2(N_q-(m-1)\sigma+p)]. \]
We aim to apply $T_{\lambda,\mu}^{N_i}$ on $(u',u'')^m$, many parcels will be eliminated by the support. If $\sigma$ is big enough, we'll have $(m-1)\sigma+p<N_i$. Moreover,
\begin{gather*}
    N_i-\big[(m-2)\sigma +2(N_q-(m-1)\sigma+p)\big] \geq N_0-\big[(m-2)\sigma +2(N_q-(m-1)\sigma+p)\big] = \\ = (m-1)\sigma +\sigma^{\frac{m-1}{m}}-(m-2)\sigma -2\big((m-1)\sigma+\sigma^{\frac{m-1}{m}}(q+1)^r-(m-1)\sigma+p\big) = \\ \geq \sigma +\sigma^{\frac{m-1}{m}} -2\Big(\sigma^{\frac{m-1}{m}}\Big(\big[\log(\sigma)\big]^3+2\Big)^r+p\Big) 
    \geq \sigma -2C_r\sigma^{\frac{m-1}{m}}\big[\log(\sigma)\big]^{3r},
\end{gather*}
for some constant $C_r>0$ depending on $r$. Then we will have $N_i>\max\text{supp}(P_0)$ if $\sigma$ is taken big enough. Furthermore, in the sum $\sum_{j=1}^{q}\sum_{l=0}^p\big(m\veps_1^{m-1}d_{l,j}'e_{N_j+l},m\veps_2^{m-1}d_{l,j}''e_{N_j+l}\big)$, we can use that $N_i>N_j+p$ when $j<i$ to conclude that all parcels indexed by $j=1,...,i-1$ have maximum support less than $N_i$. The parcel $j=i$ is the one we will use to approach $(v_1,v_2)$ and the final parcels $j=i+1,...,q$, as well as the separating term $(\veps^m_1e_{m\sigma},\veps^m_2e_{m\sigma})$, will be handled in our next calculations. All in all, we write
\begin{align*}
T_{\lambda,\mu}^{N_i}\big((u',u''&)^m\big)=(P_1',P_1'')+(P_2',P_2'')+(P_3',P_3''),
\end{align*}
where
\begin{align*}
    (P_1',P_1'')&= \sum_{l=0}^p\Big(\frac{w_{l+1}(\lambda)\cdots w_{l+N_i}(\lambda)}{w_{l+1}(\lambda_i)\cdots w_{l+N_i}(\lambda_i)}v_{1,l}e_l,\frac{w_{l+1}(\mu)\cdots w_{l+N_i}(\mu)}{w_{l+1}(\mu_i)\cdots w_{l+N_i}(\mu_i)}v_{2,l}e_l\Big),\\
    P_2'&=\sum_{j=i+1}^q\sum_{l=0}^p\frac{w_{N_j-N_i+l+1}(\lambda)\cdots w_{l+N_j}(\lambda)}{w_{l+1}(\lambda_j)\cdots w_{l+N_j}(\lambda_j)}v_{1,l}e_{N_j-N_i+l},\\
    P_2''&=\sum_{j=i+1}^q\sum_{l=0}^p\frac{w_{N_j-N_i+l+1}(\mu)\cdots w_{l+N_j}(\mu)}{w_{l+1}(\mu_j)\cdots w_{l+N_j}(\mu_j)}v_{2,l}e_{N_j-N_i+l},\\
    P_3'&=\veps_1^mw_{m\sigma-N_i+1}(\lambda)\cdots w_{m\sigma}(\lambda)e_{m\sigma-N_i},\\
    P_3''&=\veps_2^mw_{m\sigma-N_i+1}(\lambda)\cdots w_{m\sigma}(\lambda)e_{m\sigma-N_i}.
\end{align*}
The proof will be finished if we manage to prove the following:
\begin{enumerate}[(I)]
    \item $\|(P_1',P_1'')-(v_1,v_2)\|_1<\eta$;
    \item $\|(P_2',P_2'')\|_1<\eta$;
    \item $\|(P_3',P_3'')\|_1<\eta$.
\end{enumerate}
Let us begin with (I) and show that \[\|(P_1',P_1'')-(v_1,v_2)\|_1\leq \|P_1'-v_1\|_1+\|P_1''-v_2\|_1<\eta.\]
First we notice that, from  (\ref{ineq:lips}), for all $l=0,...,p$,
\begin{align*}
    \Big|\sum_{j=l+1}^{N_i+l}\log\big(w_j(\lambda)\big)-\sum_{j=l+1}^{N_i+l}\log\big(w_{j}(\lambda_i)\big)\Big| &\leq C_v \log(N_i)|\lambda-\lambda_i|\\
    &\leq C_v(b-a)\frac{\log(N_i)}{\sqrt{q}}\\
    &\leq C_v(b-a)\frac{\log(N_i)}{\sqrt{\log(\sigma)^{3}}}\xrightarrow[]{\sigma\to+\infty} 0.
\end{align*}
Then, if $\sigma$ is big enough, we get
\begin{equation}\label{eq:maj}
    \Big|\sum_{j=l+1}^{N_i+l}\log\big(w_j(\lambda)\big)-\sum_{j=l+1}^{N_i+l}\log\big(w_{j}(\lambda_i)\big)\Big| < \min\left(1, \frac{\eta}{4(p+1)\|(v_1,v_2)\|_\infty}\right).
\end{equation}
In particular, the absolute value in the left-hand side is smaller than 1, so we can use the estimate $|\exp(x)-1|\leq2|x|$ which holds for all $x\in[-1,1]$. By using (\ref{eq:maj}) we then obtain
\begin{align*}
    \|P_1'-v_1\|_1 
    &=\left\|\sum_{l=0}^p\Big(\frac{w_{l+1}(\lambda)\cdots w_{l+N_i}(\lambda)}{w_{l+1}(\lambda_i)\cdots w_{l+N_i}(\lambda_i)}-1\Big)v_{1,l}e_l\right\|_1\\
    &\leq \sum_{l=0}^p|v_{1,l}|\Bigg|\exp\Big(\sum_{j=l+1}^{N_i+l}\log\big(w_j(\lambda)\big)-\sum_{j=l+1}^{N_i+l}\log\big(w_{j}(\lambda_i)\big)\Big)-1\Bigg|\\
    &\leq \sum_{l=0}^p|v_{1,l}|\cdot 2\Big|\sum_{j=l+1}^{N_i+l}\log\big(w_j(\lambda)\big)-\sum_{j=l+1}^{N_i+l}\log\big(w_{j}(\lambda_i)\big)\Big|\\
    &< \sum_{l=0}^p|v_{1,l}|\cdot 2\frac{\eta}{4(p+1)\|v_1\|_\infty}\\
    &<\frac{\eta}{2}.
\end{align*}
Analogously we find $\|P_1''-v_2\|_1<\frac{\eta}{2}$, what verifies (I).

For (II), we first notice that, for some constant $C_v''>0$ depending on $v$,
\begin{align*}
\Big\|\sum_{l=0}^p\frac{w_{N_j-N_i+l+1}(\lambda)\cdots w_{l+N_j}(\lambda)}{w_{l+1}(\lambda_j)\cdots w_{l+N_j}(\lambda_j)}v_{1,l}e_{N_j-N_i+l}\Big\|_1
    &\leq\sum_{l=0}^p |v_{1,l}|\frac{w_{N_j-N_i+l+1}(\lambda)\cdots w_{N_j+l}(\lambda)}{w_{l+1}(\lambda_j)\cdots w_{N_j+l}(\lambda_j)} \\
    &\leq C_v''\frac{(N_j+l)^{\lambda-\lambda_j}}{(N_j-N_i+l)^\lambda}.
    \end{align*}
We plan to show that this is the general term of a convergent series on $j$ whose sum can be made small by increasing $\sigma$. As we are dealing with a bi-dimensional partition, it is necessary to keep track of the sign of $|\lambda-\lambda_j|$. If $\lambda_j\geq \lambda$ then, just as in one dimension, we have $(N_j+l)^{\lambda-\lambda_j}\leq 1$, that is, our problem reduces to estimate \[\frac{1}{(N_j-N_i+l)^\lambda},\]
what is included in the calculations below. 
Let us now consider the more difficult case $\lambda_j<\lambda$, which only occurs in two (or more) dimensions. We have
    \begin{align*}
    \frac{(N_j+l)^{\lambda-\lambda_j}}{(N_j-N_i+l)^\lambda}&=\frac{\big((m-1)\sigma+\sigma^{\frac{m-1}{m}}(j+1)^r+l\big)^{\lambda-\lambda_j}}{\Big(\big((j+1)^r-(i+1)^r\big)\sigma^{\frac{m-1}{m}}+l\Big)^\lambda}\\
    &\overset{l\geq 0}{\leq}\frac{\sigma^{\lambda-\lambda_j}(j+1)^{r(\lambda-\lambda_j)}\big(\frac{(m-1)}{(j+1)^r}+\sigma^{\frac{-1}{m}}+\frac{l}{\sigma(j+1)^r}\big)^{\lambda-\lambda_j}}{\sigma^{\lambda\frac{m-1}{m}}(j+1)^{r\lambda}\big(1-\frac{(i+1)^r}{(j+1)^r}\big)^\lambda}\\
    &\overset{r\geq 1}{\leq}\frac{\big(\frac{(m-1)}{(j+1)^r}+\sigma^{\frac{-1}{m}}+\frac{l}{\sigma(j+1)^r}\big)^{\lambda-\lambda_j}}{\sigma^{\lambda_j -\frac{\lambda}{m} }(j+1)^{r\lambda_j}\big(1-\frac{i+1}{j+1}\big)^\lambda}\\
    &\!\!\overset{j\geq i+1}{\leq} \frac{C_0}{\sigma^{\lambda_j -\frac{\lambda}{m} }(j+1)^{r\lambda_j}\big(1-\frac{i+1}{i+2}\big)^\lambda}\\
    &\overset{i\leq q}{\leq} \frac{C_0}{(j+1)^{r\lambda_j}\Big(\frac{\sigma^{\frac{\lambda_j -\frac{\lambda}{m} }{\lambda}}}{q+2}\Big)^\lambda}\\
    &\!\!\!\!\!\!\!\!\!\! \overset{q\leq [\log(\sigma)]^3+1}{\leq}{ \frac{C_0}{(j+1)^{r\lambda_j}\Big(\frac{\sigma^{c'}}{(\log\sigma)^3+3}\Big)^\lambda}},
\end{align*}
for some constant $C_0$ depending on $a,b,r, m, p$ and $c'=\frac{a'}{b'}-\frac{1}{m}\leq \frac{\lambda_j}{\lambda}-\frac{1}{m}$ is positive since $m\geq 2$ and $2a'\geq b'$. These estimates show that 
\begin{align*}
    \|P_2'\|_1&=\left\|\sum_{j=i+1}^q\sum_{l=0}^p\frac{w_{N_j-N_i+l+1}(\lambda)\cdots w_{l+N_j}(\lambda)}{w_{l+1}(\lambda_j)\cdots w_{l+N_j}(\lambda_j)}v_{1,l}e_{N_j-N_i+l}\right\|_1\\
    &\leq \frac{C_v''C_0}{\Big(\frac{\sigma^{c'}}{(\log\sigma)^3+3}\Big)^\lambda}\sum_{j=i+1}^q\frac{1}{(j+1)^{ra'}} \xrightarrow[]{\sigma\to+\infty}0.
\end{align*}
Hence, $\|P_2'\|_1<\frac{\eta}{2}$ if $\sigma$ is big enough. Analogously we get $\|P_2''\|_1< \frac{\eta}{2}$, what shows (II).

Let us proceed to the proof of (III). We have
\begin{align*}
    \|P_3'\|_1&=\|\veps_1^mw_{m\sigma-N_i+1}(\lambda)\cdots w_{m\sigma}(\lambda)e_{m\sigma-N_i}\|_1\\
    &\leq\frac{w_{m\sigma-N_q+1}(b')\cdots w_{m\sigma}(b')}{w_{1}(a')\cdots w_{m\sigma}(a')}\\ 
    &=\frac{w_{1}(b')\cdots w_{m\sigma}(b')}{w_{1}(a')\cdots w_{m\sigma}(a')}\times \frac{1}{w_{1}(b')\cdots w_{m\sigma-N_q}(b')}\\
    &\leq \frac{(m\sigma)^{b'-a'}}{(m\sigma-N_q)^{b'}}\\
    &\leq{ \frac{(m\sigma)^{b'-a'}}{(\sigma-\sigma^{\frac{m-1}{m}}((\log\sigma)^3+2)^r)^{b'}}}\xrightarrow[]{\sigma\to+\infty}0.
\end{align*}
Hence, $\|P_3'\|_1< \frac{\eta}{2}$ if $\sigma$ is big enough. Analogously, $\|P_3''\|_1< \frac{\eta}{2}$ and the proof is done.
\end{proof}

We finish this paper by returning to the one dimensional framework. Let $I$ be an interval and let $\lambda\in I\mapsto \big(w_n(\lambda)\big)_{n}$ be a continuously parametrized family of weights such that $\lambda\mapsto \sum_{k=1}^n \log\big(w_k(\lambda)\big)$ is $F(n)$-Lipschitz for some function $F:\NN\to\NN$. As exemplified by its authors, \cite[Theorem 3.12]{BCP2} can be applied for some cases where $F(n)=n$. It turns out that it can be more generally applied for the case where $w_1(\lambda)\cdots w_n(\lambda)=\exp(\lambda n^\alpha)$ with $\alpha\in (0,1]$, which is a particular case where $F(n)=n^\alpha$. However, the same result doesn't work for the case $F(n)=\log(n)$, although $w_n(\lambda)=1+\frac{\lambda}{n}$ does admit a common hypercyclic algebra as we proved here (through similar, yet different, calculations). 
\begin{question}
Could one obtain a unified one dimensional criterion on $F(n)$ granting a common hypercyclic subalgebra of $\ell_1(\NN)$ for the convolution product and including $F(n)=n^\alpha, \alpha\leq 1,$ and $F(n)=\log(n)$ as particular cases?
\end{question}

\textbf{Acknowledgement.} The author would like to thank the referee for their thoughtful comments and efforts towards improving the manuscript.


\end{document}